\documentclass[reqno, english]{amsart}
\usepackage{etex}
\usepackage{amsmath,amssymb,amsthm,bbm,mathtools,comment}
\usepackage[shortlabels]{enumitem}
\usepackage[pdftex,colorlinks,backref=page,citecolor=blue]{hyperref}
\hypersetup{pdfpagemode=UseNone,pdfstartview={XYZ null null 1.00}}
\usepackage[mathscr]{euscript}
\usepackage[usenames,dvipsnames]{color}
\usepackage{adjustbox,tikz,calc,graphics,babel,standalone}
\usetikzlibrary{shapes.misc,calc,intersections,patterns,decorations.pathreplacing}
\usetikzlibrary{arrows,shapes,positioning}
\usetikzlibrary{decorations.markings}
\usepackage[final]{microtype}
\usepackage[numbers]{natbib}
\usepackage{cmtiup}
\usepackage{amsfonts}
\usepackage{graphicx}
\usepackage{caption}
\usepackage{subcaption}
\usepackage{verbatim}
\usepackage{array}
\usepackage[frame,cmtip,arrow,matrix,line,graph,curve]{xy}
\usepackage{graphpap, color, pstricks}
\usepackage{pifont}
\usepackage[final]{microtype}
\usepackage{cmtiup}
\usepackage{todonotes}

\allowdisplaybreaks

\theoremstyle{plain}
\newtheorem{theorem}{Theorem}[section]		
\newtheorem{lemma}[theorem]{Lemma}

\newtheorem{proposition}[theorem]{Proposition}
\newtheorem{corollary}[theorem]{Corollary}
\newtheorem{conjecture}[theorem]{Conjecture}
\newtheorem{problem}[theorem]{Problem}
\newtheorem{definition}[theorem]{Definition}
\theoremstyle{remark}

\def\FF{\mathcal{F}}
\def\HH{\mathcal{H}}
\def\GG{\mathcal{G}}
\def\E{\mathbb{E}}

\newcommand{\eps}{\ensuremath{\varepsilon}}
\let\emptyset\varnothing

\let\originalleft\left
\let\originalright\right
\renewcommand{\left}{\mathopen{}\mathclose\bgroup\originalleft}
\renewcommand{\right}{\aftergroup\egroup\originalright}

\makeatletter
\def\imod#1{\allowbreak\mkern10mu({\operator@font mod}\,\,#1)}
\makeatother

\title{Tiling with monochromatic bipartite graphs of bounded maximum degree}

\author{Ant\'onio Gir\~ao}
\thanks{AG: Institut f\"ur Informatik, Universit\"at Heidelberg, Germany. E-mail: {\nolinkurl{a.girao@informatik.uni-heidelberg.de}}. Research supported by Deutsche Forschungsgemeinschaft (DFG, German Research Foundation) under Germany’s Excellence Strategy EXC-2181/1 - 390900948 (the Heidelberg STRUCTURES Cluster of Excellence)}
\author{Oliver Janzer}
\thanks{OJ: Department of Mathematics, ETH Z\"urich, Switzerland. Email:  {\nolinkurl{oliver.janzer@math.ethz.ch}}.
Research supported by an ETH Z\"urich Postdoctoral Fellowship 20-1 FEL-35.}

\begin{document}

\maketitle

\begin{abstract}
    We prove that for any $r\in \mathbb{N}$, there exists a constant $C_r$ such that the following is true. Let $\mathcal{F}=\{F_1,F_2,\dots\}$ be an infinite sequence of bipartite graphs such that $|V(F_i)|=i$ and $\Delta(F_i)\leq \Delta$ hold for all~$i$. Then in any $r$-edge coloured complete graph $K_n$, there is a collection of at most $\exp(C_r\Delta)$ monochromatic subgraphs, each of which is isomorphic to an element of $\mathcal{F}$, whose vertex sets partition $V(K_n)$. This proves a conjecture of Corsten and Mendon\c{c}a in a strong form and generalizes results on the multicolour Ramsey numbers of bounded-degree bipartite graphs.
\end{abstract}

\section{Introduction}

The problem of partitioning the vertex set of edge-coloured complete graphs into a small number of certain monochromatic pieces has a very rich history; see \cite{Gyar16} for a recent survey. An early example of a problem of this kind is Lehel's conjecture \cite{Ayel79}. The conjecture states that any 2-edge coloured complete graph $K_n$ contains a red and a blue cycle whose vertex sets partition $V(K_n)$. Here, the empty graph, singletons and edges are regarded as cycles. This conjecture was proved by \L uczak, R\"odl and Szemer\'edi \cite{LRSz98} for large $n$. Later, Allen~\cite{All08} significantly improved the bound on $n$. Finally, Bessy and Thomass\'e \cite{BT10} proved Lehel's conjecture for all $n$.

Regarding more colours, Erd\H os, Gy\'arf\'as and Pyber \cite{EGyP91} proved that the vertex set of any $r$-edge coloured complete graph can be partitioned to $O(r^2\log r)$ monochromatic cycles and conjectured that in fact $r$ monochromatic cycles should be enough. However, this was disproved by Pokrovskiy \cite{Pok14}. The current best known upper bound is $O(r\log r)$, due to Gy\'arf\'as, Ruszink\'o, S\'ark\"ozy and Szemer\'edi \cite{GyRSSz06}.

Grinshpun and S\'ark\"ozy \cite{GS16} considered the more general problem where we want to partition the vertex set of the complete graph to graphs chosen from a fixed bounded-degree sequence. Let $\FF=\{F_1,F_2,\dots\}$ be a sequence of graphs. A \emph{monochromatic $\FF$-tiling} of size $s$ of an edge coloured complete graph $K_n$ is a collection of $s$ monochromatic subgraphs, each isomorphic to an element of $\mathcal{F}$, whose vertex sets partition $V(K_n)$. The \emph{$r$-colour tiling number} of $\FF$, denoted by $\tau_r(\FF)$, is the minimal $s$ such that every $r$-edge coloured complete graph has a monochromatic $\FF$-tiling of size at most $s$ (if no such $s$ exists, then we let $\tau_r(\FF)=\infty$). Call $\mathcal{F}=\{F_1,F_2,\dots\}$ a \emph{$\Delta$-bounded graph sequence} if $v(F_i)=i$ and $\Delta(F_i)\leq \Delta$ for all $i$. (Here and below, $v(G)$ denotes the number of vertices in $G$ and $\Delta(G)$ denotes the maximum degree of $G$.) Also, call $\mathcal{F}=\{F_1,F_2,\dots\}$ a \emph{bipartite $\Delta$-bounded graph sequence} if in addition each $F_i$ is bipartite.

Grinshpun and S\'ark\"ozy gave an upper bound for the $2$-colour tiling number of $\Delta$-bounded graph sequences.

\begin{theorem}[Grinshpun--S\'ark\"ozy \cite{GS16}] \label{thm:2-colour tiling}
    There exists an absolute constant $C$ such that for any $\Delta\geq 2$ and any $\Delta$-bounded graph sequence $\FF$, $\tau_2(\FF)\leq \exp(C\Delta \log \Delta)$.
\end{theorem}

They also gave a better bound for the $2$-colour tiling number of bipartite $\Delta$-bounded sequences.

\begin{theorem}[Grinshpun--S\'ark\"ozy \cite{GS16}] \label{thm:2-colour bipartite tiling}
    There exists an absolute constant $C$ such that for any bipartite $\Delta$-bounded graph sequence $\FF$, $\tau_2(\FF)\leq \exp(C\Delta)$.
\end{theorem}

Moreover, they showed that Theorem \ref{thm:2-colour bipartite tiling} is tight.

\begin{theorem}[Grinshpun--S\'ark\"ozy \cite{GS16}] \label{thm:tiling lower bound}
    There is an absolute constant $c>0$ such that for any $\Delta$ there exists a bipartite $\Delta$-bounded graph sequence $\FF$ with $\tau_2(\FF)\geq \exp(c\Delta)$.
\end{theorem}

Regarding more colours, they made the following conjecture.

\begin{conjecture}[Grinshpun--S\'ark\"ozy \cite{GS16}] \label{con:general tiling}
    For every positive integer $r$ there exists a constant $C_r$ such that for any $\Delta\geq 2$ and any $\Delta$-bounded graph sequence $\FF$, $\tau_r(\FF)\leq \exp(\Delta^{C_r})$.
\end{conjecture}

Recently, Corsten and Mendon\c{c}a established the finiteness of $\tau_r(\FF)$ for $\Delta$-bounded graph sequences and proved a triple exponential upper bound in $\Delta$.

\begin{theorem}[Corsten--Mendon\c{c}a \cite{CM21}] \label{thm:r-colour tiling}
    There exists an absolute constant $C$ such that for every $\Delta$-bounded graph sequence $\FF$, $\tau_r(\FF)\leq \exp\left(\exp\big(r^{Cr\Delta^3}\big)\right)$.
\end{theorem}

They point out that their proof gives a double exponential upper bound for bipartite $\Delta$-bounded graph sequences. Moreover, they write that it would be very interesting to prove Conjecture \ref{con:general tiling} for bipartite $\Delta$-bounded graph sequences. In this paper, we prove such a result with a stronger bound.

\begin{theorem} \label{thm:tiling with bipartite sequence}
    For every positive integer $r$ there exists a constant $C_r$ such that for any $\Delta$ and any bipartite $\Delta$-bounded graph sequence $\FF$, $\tau_r(\FF)\leq \exp(C_r\Delta)$.
\end{theorem}

By Theorem \ref{thm:tiling lower bound}, this result is tight up to the value of $C_r$.

\subsection{Connection to the Ramsey numbers of bounded-degree graphs}

The results just mentioned are closely related to the study of the Ramsey numbers of bounded-degree graphs. The research on these Ramsey numbers was initiated by Burr and Erd\H os \cite{BE75}. They conjectured that for any $\Delta$ there is a constant $c(\Delta)$ such that the Ramsey number of every graph $H$ with $n$ vertices and maximum degree at most $\Delta$ satisfies $R(H)\leq c(\Delta)n$. The conjecture was proved by Chv\'atal, R\"odl, Szemer\'edi and Trotter \cite{CRSzT83} using Szemer\'edi's regularity lemma. There has been plenty of research on improving the value of $c(\Delta)$ since. First, Eaton \cite{Eat98} showed that $c(\Delta)\leq 2^{2^{C\Delta}}$ for some absolute constant $C$. The bound was further improved by Graham, R\"odl and Ruci\'nski \cite{GRR00} who proved that $c(\Delta)\leq 2^{C\Delta \log^2 \Delta}$. Finally, Conlon, Fox and Sudakov \cite{CFS12} showed that $c(\Delta)\leq 2^{C\Delta \log \Delta}$, which is the current best bound, although it is conjectured that $c(\Delta)\leq 2^{C\Delta}$.

When $H$ is bipartite, better bounds are known. Improving on earlier work by Graham, R\"odl and Ruci\'nski~\cite{GRR01}, Conlon \cite{Con09}, and Fox and Sudakov \cite{FS09}, Conlon, Fox and Sudakov \cite{CFS16} showed that when $H$ is bipartite, we can take $c(\Delta)\leq 2^{\Delta+6}$. On the other hand, Graham, R\"odl and Ruci\'nski \cite{GRR00,GRR01} proved that for every $\Delta$ and sufficiently large $n$ there are bipartite graphs $H$ with $n$ vertices and maximum degree $\Delta$ for which $R(H)\geq 2^{c'\Delta}n$.

The best known upper bound for the $r$-colour Ramsey number of an $n$-vertex graph $H$ with maximum degree $\Delta$ is $R_r(H)\leq \exp(C_r\Delta^2)n$ (see \cite{CFS15}). Fox and Sudakov \cite{FS09} showed that if $H$ is also bipartite, then $R_r(H)\leq \exp(C_r\Delta)n$.

Note that our Theorem \ref{thm:tiling with bipartite sequence} can be viewed as a generalization of the last bound $R_r(H)\leq \exp(C_r\Delta)n$. Indeed, let $H$ be an $n$-vertex bipartite graph with maximum degree at most $\Delta$. Clearly, we can define a bipartite $\Delta$-bounded graph sequence $\FF=\{F_1,F_2,\dots\}$ where $F_n=H$ and $F_n$ is a subgraph of $F_m$ for every $m>n$. Now let $r\in \mathbb{N}$ and let $C_r$ be the constant provided by Theorem \ref{thm:tiling with bipartite sequence}. We claim that if $N\geq \exp(C_r\Delta)n$, then any $r$-edge colouring of $K_N$ contains a monochromatic copy of $H$. To see this, note that by Theorem~\ref{thm:tiling with bipartite sequence}, any such colouring has a monochromatic $\mathcal{F}$-tiling of size at most $\exp(C_r\Delta)$. In particular, it contains a monochromatic copy of $F_m$ for some $m\geq N/\exp(C_r\Delta)\geq n$. Any such $F_m$ contains $H$ as a subgraph, so the colouring contains a monochromatic copy of $H$.

Similarly, Theorems \ref{thm:2-colour tiling}, \ref{thm:2-colour bipartite tiling} and \ref{thm:r-colour tiling} generalize the corresponding upper bounds on Ramsey numbers. On the other hand, it is clear that the Ramsey bounds do not directly imply tiling results.

\subsection{Sketch of the proof of Theorem \ref{thm:tiling with bipartite sequence}}

Before we turn to the proof, we provide a brief outline. For each main step of the argument, we shall also give a reference to the corresponding subsection in the proof. As is common in tiling problems, we use the absorption method.

Let us give a rough sketch how the absorption is performed. Let $\FF=\{F_1,F_2,\dots\}$ be a bipartite $\Delta$-bounded graph sequence and assume that our $r$-edge coloured complete graph $G$ contains a monochromatic red subgraph $F$, isomorphic to an element of $\FF$. Let $(X,Y)$ be a bipartition of $F$ and let $Z$ be a set of vertices in $G$, disjoint from $X\cup Y$. Suppose that $Z$ is not much larger than $Y$ and that for each $z\in Z$ there are many $y\in Y$ such that all edges between $z$ and the set $N_F(y)$ are red. The key observation is that if the graph $F$ has a certain structure, then it can absorb $Z$; that is, $G$ has a small collection of pairwise vertex-disjoint monochromatic subgraphs, each isomorphic to an element of $\FF$, whose union is precisely $X\cup Y\cup Z$. Suppose that for any $u,y\in Y$, there are many disjoint choices for $(v,w)\in Y^2$ such that the edges between $v$ and $N_F(u)$, the edges between $w$ and $N_F(v)$ and the edges between $y$ and $N_F(w)$ are all red. By repeatedly using the upper bound for the multicolour Ramsey number of bounded-degree bipartite graphs (e.g. from \cite{FS09}), we can cover almost all of $Z$ by a small number of monochromatic subgraphs, each isomorphic to an element of $\FF$. Since $Z$ is not much larger than $Y$, the remaining uncovered subset $Z'\subset Z$ eventually becomes much smaller than $Y$. Let $Z'=\{z_1,\dots,z_k\}$. Since $Z'$ is much smaller than $Y$ and for every $z\in Z$ there are many $y\in Y$ such that all edges between $z$ and $N_F(y)$ are red, we can choose distinct $y_1,y_2,\dots,y_k\in Y$ such that for each $i\in [k]$, all edges between $z_i$ and $N_F(y_i)$ are red. Since $k=|Z'|$ is much smaller than $|Y|$, by the upper bound on Ramsey numbers there exists a monochromatic copy of $F_k$ in $G[Y]$. Let the vertex set of this copy be $\{u_1,u_2,\dots,u_k\}$. Using the assumed property of $Y$, we can choose distinct vertices $v_1,v_2,\dots,v_k,w_1,w_2,\dots,w_k$ such that for each $i\in [k]$, the edges between $v_i$ and $N_F(u_i)$, the edges between $w_i$ and $N_F(v_i)$ and the edges between $y_i$ and $N_F(w_i)$ are all red. Therefore replacing in $F$, for each $i\in [k]$, $u_i$ by $v_i$, $v_i$ by $w_i$, $w_i$ by $y_i$ and $y_i$ by $z_i$, we get a monochromatic red subgraph of $G$ isomorphic to $F$ with vertex set $X\cup (Y\setminus \{u_1,\dots,u_k\})\cup \{z_1,\dots,z_k\}$ (see Figure \ref{fig:switching} for an illustration of how the graph $F$ is modified). We had already found a monochromatic graph with vertex set $\{u_1,\dots,u_k\}$ and a small collection of monochromatic subgraphs partitioning $Z\setminus \{z_1,\dots,z_k\}$, so we have partitioned $X\cup Y\cup Z$ with a small number of monochromatic subgraphs, each isomorphic to an element of $\FF$.

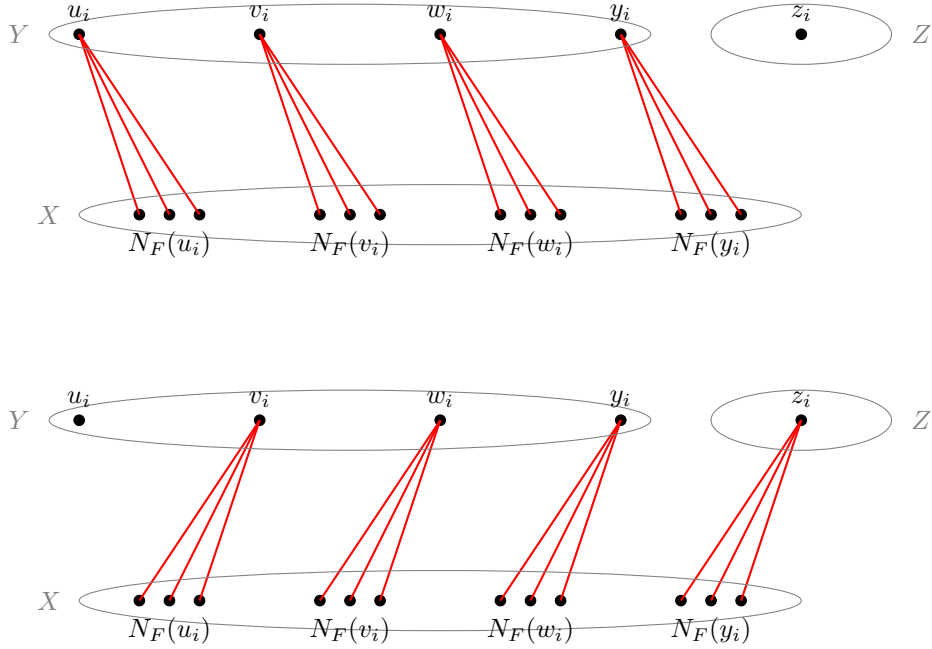
\begin{figure}
    \centering
    \begin{tikzpicture}[scale=0.4]
        \draw[fill=black](2,0)circle(5pt);
        \draw[fill=black](3,0)circle(5pt);
        \draw[fill=black](4,0)circle(5pt);
        \draw[fill=black](8,0)circle(5pt);
        \draw[fill=black](9,0)circle(5pt);
        \draw[fill=black](10,0)circle(5pt);
        \draw[fill=black](14,0)circle(5pt);
        \draw[fill=black](15,0)circle(5pt);
        \draw[fill=black](16,0)circle(5pt);
        \draw[fill=black](20,0)circle(5pt);
        \draw[fill=black](21,0)circle(5pt);
        \draw[fill=black](22,0)circle(5pt);
        
        \draw[fill=black](0,6)circle(5pt);
        \draw[fill=black](6,6)circle(5pt);
        \draw[fill=black](12,6)circle(5pt);
        \draw[fill=black](18,6)circle(5pt);
        \draw[fill=black](24,6)circle(5pt);
        
        \draw[thick,red](0,6)--(2,0);
        \draw[thick,red](0,6)--(3,0);
        \draw[thick,red](0,6)--(4,0);
        \draw[thick,red](6,6)--(8,0);
        \draw[thick,red](6,6)--(9,0);
        \draw[thick,red](6,6)--(10,0);
        \draw[thick,red](12,6)--(14,0);
        \draw[thick,red](12,6)--(15,0);
        \draw[thick,red](12,6)--(16,0);
        \draw[thick,red](18,6)--(20,0);
        \draw[thick,red](18,6)--(21,0);
        \draw[thick,red](18,6)--(22,0);
        
        \draw[rotate around={90:(12,0)},gray] (12,0) ellipse (1 and 12);
        \draw[rotate around={90:(9,6)},gray] (9,6) ellipse (1 and 10);
        \draw[rotate around={90:(24,6)},gray] (24,6) ellipse (1 and 3);
        
        \node at (3,-1) {$N_F(u_i)$};
        \node at (9,-1) {$N_F(v_i)$};
        \node at (15,-1) {$N_F(w_i)$};
        \node at (21,-1) {$N_F(y_i)$};
        \node[gray] at (-1,0) {$X$};
        \node[gray] at (-2,6) {$Y$};
        \node[gray] at (28,6) {$Z$};
        \node at (0,6.7) {$u_i$};
        \node at (6,6.7) {$v_i$};
        \node at (12,6.7) {$w_i$};
        \node at (18,6.7) {$y_i$};
        \node at (24,6.7) {$z_i$};
    \end{tikzpicture}
    
    \vspace{15mm}
    
    \begin{tikzpicture}[scale=0.4]
        \draw[fill=black](2,0)circle(5pt);
        \draw[fill=black](3,0)circle(5pt);
        \draw[fill=black](4,0)circle(5pt);
        \draw[fill=black](8,0)circle(5pt);
        \draw[fill=black](9,0)circle(5pt);
        \draw[fill=black](10,0)circle(5pt);
        \draw[fill=black](14,0)circle(5pt);
        \draw[fill=black](15,0)circle(5pt);
        \draw[fill=black](16,0)circle(5pt);
        \draw[fill=black](20,0)circle(5pt);
        \draw[fill=black](21,0)circle(5pt);
        \draw[fill=black](22,0)circle(5pt);
        
        \draw[fill=black](0,6)circle(5pt);
        \draw[fill=black](6,6)circle(5pt);
        \draw[fill=black](12,6)circle(5pt);
        \draw[fill=black](18,6)circle(5pt);
        \draw[fill=black](24,6)circle(5pt);
        
        \draw[thick,red](6,6)--(2,0);
        \draw[thick,red](6,6)--(3,0);
        \draw[thick,red](6,6)--(4,0);
        \draw[thick,red](12,6)--(8,0);
        \draw[thick,red](12,6)--(9,0);
        \draw[thick,red](12,6)--(10,0);
        \draw[thick,red](18,6)--(14,0);
        \draw[thick,red](18,6)--(15,0);
        \draw[thick,red](18,6)--(16,0);
        \draw[thick,red](24,6)--(20,0);
        \draw[thick,red](24,6)--(21,0);
        \draw[thick,red](24,6)--(22,0);
        
        \draw[rotate around={90:(12,0)},gray] (12,0) ellipse (1 and 12);
        \draw[rotate around={90:(9,6)},gray] (9,6) ellipse (1 and 10);
        \draw[rotate around={90:(24,6)},gray] (24,6) ellipse (1 and 3);
        
        \node at (3,-1) {$N_F(u_i)$};
        \node at (9,-1) {$N_F(v_i)$};
        \node at (15,-1) {$N_F(w_i)$};
        \node at (21,-1) {$N_F(y_i)$};
        \node[gray] at (-1,0) {$X$};
        \node[gray] at (-2,6) {$Y$};
        \node[gray] at (28,6) {$Z$};
        \node at (0,6.7) {$u_i$};
        \node at (6,6.7) {$v_i$};
        \node at (12,6.7) {$w_i$};
        \node at (18,6.7) {$y_i$};
        \node at (24,6.7) {$z_i$};
    \end{tikzpicture}
    
    \caption{Modification of the subgraph $F$, resulting in an isomorphic red subgraph with vertex set $X\cup (Y\setminus \{u_1,\dots,u_k\})\cup \{z_1,\dots,z_k\}$}
    \label{fig:switching}
\end{figure}

More generally, we can use a similar argument if there exist large pairwise disjoint subsets $Y_1,\dots,Y_t\subset Y$ which together almost cover $Y$ and which have the property that for each $j\in [t]$ and every $u,y\in Y_j$ there are many disjoint choices for $(v,w)\in Y_j^2$ such that the edges between $v$ and $N_F(u)$, the edges between $w$ and $N_F(v)$ and the edges between $y$ and $N_F(w)$ are all red.

The precise description of the absorption is given in Subsection \ref{sec:absorption process}. Subsections \ref{sec:construct absorber} and \ref{sec:good subgraph} are devoted to constructing the absorbers. More specifically, the construction of the subsets $Y_1,\dots,Y_t$ with the above key property will rely on the results in Subsection~\ref{sec:construct absorber}, while the construction of the subgraph $F$ will be presented in Subsection~\ref{sec:good subgraph}.

For the latter, we build upon the technique of \cite{FS09} and \cite{Con09} that was used for finding a large red subgraph isomorphic to some $F\in \FF$ in a host graph $G$ with many red edges. The idea was to use dependent random choice to find a large subset $U\subset V(G)$ with the property that all but a tiny proportion of the sets of size $\Delta$ in $U$ have many common red neighbours. Using the condition that $F$ is bipartite and has maximum degree at most $\Delta$, this suffices for finding a red copy of $F$ in $G$ with parts $X\subset U$ and $Y$. For our results, we will need to tweak their argument slightly as we need to find a copy of $F$ which also satisfies some additional properties, such as that for every vertex $w\in V(G)$ which has many red neighbours in $U$, there should be many vertices $y\in Y$ such that all edges between $w$ and $N_F(y)$ are red. Fortunately, there are many choices to find a red copy of $F$ with $X\subset U$, so we will be able to argue that a random embedding will possess the required properties. The precise statement is given in Lemma \ref{lem:find good subgraph}.

In Subsection \ref{sec:completing the proof}, we complete the proof of our main result, Theorem \ref{thm:tiling with bipartite sequence}.

\noindent \textbf{Notation and remarks.} We will use the standard asymptotic notations $O$ and $\Omega$. For positive functions $f$ and $g$, $f=O(g)$ (respectively, $f=\Omega(g)$) means that there exists a positive absolute constant $C$ such that $f\leq Cg$ (respectively, $f\geq Cg$). Also, $f=O_r(g)$ and $f=\Omega_r(g)$ mean that the same inequalities hold for some positive $C$ which only depends on $r$.

Very occasionally, we treat real numbers as integers when doing so makes no significant difference in the argument.

For the whole Section \ref{sec:proof}, we fix positive integers $r$ and $\Delta$ and a $\Delta$-bounded bipartite graph sequence $\FF=\{F_1,F_2,\dots\}$ (and we do not define them in each lemma).

\section{The proof of Theorem \ref{thm:tiling with bipartite sequence}} \label{sec:proof}

\subsection{Preliminaries} \label{sec:preliminaries}

In this subsection we present a few preliminary lemmas which will be used in our proofs. The first one is a density version of the upper bound on the multicolour Ramsey numbers of bipartite bounded-degree graphs.

\begin{proposition}[Fox--Sudakov \cite{FS09}] \label{prop:density ramsey}
    Let $F$ be a bipartite graph with $k$ vertices and maximum degree $\Delta\geq 1$. If $\eps > 0$ and $G$ is a graph with $n \geq 32\Delta \eps^{-\Delta}k$ vertices and at least $\eps \binom{n}{2}$ edges, then $F$ is a subgraph of $G$.
\end{proposition}

\begin{corollary} \label{cor:big mono graph}
    If $n\geq 32\Delta r^{\Delta} k$, then any $r$-edge colouring of $K_n$ contains a monochromatic copy of $F_k$.
\end{corollary}

\begin{proof}
    There is a colour which appears on at least $\frac{1}{r}\binom{n}{2}$ edges, and we can apply Proposition \ref{prop:density ramsey} with $\eps=1/r$ for the graph with these edges.
\end{proof}

By repeatedly finding large monochromatic subgraphs, we can cover almost all vertices with a small number of monochromatic subgraphs from $\FF$.

\begin{corollary} \label{cor:mono cover}
    For any $0<t\leq n$ and any $r$-edge colouring of $K_n$, there exists a collection of at most $64\Delta r^{\Delta}(\log (n/t)+2)$ pairwise vertex-disjoint monochromatic subgraphs, each of which is isomorphic to an element of $\mathcal{F}$, whose union covers more than $n-t$ vertices.
    
    In particular, any $r$-edge colouring of $K_n$ has a monochromatic $\FF$-tiling of size at most $64\Delta r^{\Delta}(\log n+2)$.
\end{corollary}

\begin{proof}
We find suitable subgraphs by the following algorithm. Let us assume that we have already covered all but at most $s$ vertices of $K_n$. If $s\leq 64\Delta r^{\Delta}$, then cover these vertices by singletons (note that the unique $1$-vertex graph belongs to $\mathcal{F}$); this gives a cover of all vertices. Else, there exists a positive integer $\ell$ such that $\frac{s}{64\Delta r^{\Delta}}\leq \ell \leq \frac{s}{32\Delta r^{\Delta}}$. Since there are at least $s\geq 32\Delta r^{\Delta} \ell$ vertices not yet covered, by Corollary~\ref{cor:big mono graph} there exists a monochromatic copy of $F_\ell$ whose vertices were not yet covered. Add this to the collection of covering subgraphs.

Eventually, we will have covered all but fewer than $t$ vertices, and then we stop. We claim that we have used at most $64\Delta r^{\Delta}(\log (n/t)+2)$ subgraphs. To prove this, it suffices to show that before we terminate the process or get to the stage where $s\leq 64\Delta r^{\Delta}$, we used at most $64\Delta r^{\Delta}(\log (n/t)+1)$ subgraphs. Suppose that we used $k$ subgraphs to get to this stage. Then, since each time we covered at least $\frac{1}{64\Delta r^{\Delta}}$ proportion of the yet uncovered vertices, it follows that $t\leq (1-\frac{1}{64\Delta r^{\Delta}})^{k-1}n$. Using the inequality $1-x\leq e^{-x}$, we obtain that $k\leq 64\Delta r^{\Delta}\log (n/t)+1$, completing the proof of the first assertion. The second assertion follows from the first one by taking $t=1$.
\end{proof}

The next lemma is a version of dependent random choice.

\begin{lemma}\label{lem:dependent random choice}
Let $k,t\in \mathbb{N}$ and $0<\eps,\delta,\gamma<1$ such that $\delta\eps^{kt}\geq 2\gamma^t$. Let $G=(A,B)$ be a bipartite graph with $e(G)\geq \eps |A||B|$.
Then there is a set $S\subset A$ of size at least $\frac{1}{2}\eps^t|A|$ such that all but at most $\delta |S|^{k}$ sets of $k$ vertices in $S$ have at least $\gamma|B|$ common neighbours. 
\end{lemma}

\begin{proof}
Choose $t$ vertices from $B$ at random with replacement and call the set of these vertices $T$. Let $S=N(T)$, the common neighbourhood of the vertices in $T$.

Now
$$\E[|S|]=\sum_{v\in A} (d(v)/|B|)^t\geq |A|(\bar{d}/|B|)^t\geq |A|\eps^t,$$
where $\bar{d}$ denotes the average degree of the vertices in $A$. Hence, $\E[|S|^{k}]\geq \E[|S|]^{k}\geq \eps^{k t}|A|^{k}$.

Write $X$ for the number of sets of $k$ vertices in $S$ which have fewer than $\gamma|B|$ common neighbours.
Note that $\E[X]\leq \binom{|A|}{k}\gamma^t\leq |A|^{k}\gamma^t$. Hence,
$$\E[\delta |S|^{k}-X]\geq \delta \eps^{kt}|A|^k-\gamma^{t} |A|^k.$$

In particular, there exists an outcome with $\delta |S|^{k}-X\geq \delta \eps^{kt}|A|^k-\gamma^{t} |A|^k$. Since $\delta\eps^{kt}\geq 2\gamma^t$, we get $\delta |S|^{k}-X\geq \frac{1}{2}\delta \eps^{kt}|A|^k$. This means on the one hand that $X\leq \delta|S|^k$, so at most $\delta|S|^k$ sets of $k$ vertices in $S$ have fewer than $\gamma|B|$ common neighbours. On the other hand, $\delta|S|^k\geq \frac{1}{2}\delta \eps^{kt}|A|^k$, so $|S|\geq \frac{1}{2^{1/k}}\eps^t|A|\geq \frac{1}{2}\eps^t|A|$.
\end{proof}

The next two results are specializations of the previous one and are suited specifically to our needs.

\begin{lemma} \label{lem:k-set drc}
    There is a constant $C=C(r)$ such that the following is true. Let $k\in \mathbb{N}$ and $0<\delta<1/2$. Let $G=(A,B)$ be a bipartite graph with $e(G)\geq \frac{1}{r} |A||B|$. Then there is a set $S\subset A$ of size at least $\delta^{C}|A|$ such that all but at most $\delta |S|^{k}$ sets of $k$ vertices in $S$ have at least $\frac{(1/r)^k}{2}|B|$ common neighbours.
\end{lemma}

\begin{proof}
Let $\eps=\frac{1}{r}$, $\gamma=\frac{(1/r)^k}{2}$ and choose a positive integer $t$ such that $2^{t-2}\leq 1/\delta<2^{t-1}$. Then it is easy to see that $\delta\eps^{kt}\geq 2\gamma^t$. Hence, by Lemma \ref{lem:dependent random choice}, there exists a set $S\subset A$ of size at least $\frac{1}{2}\eps^t|A|$ such that all but at most $\delta |S|^{k}$ sets of $k$ vertices in $S$ have at least $\frac{(1/r)^k}{2}|B|$ common neighbours. Since $\frac{1}{2}\eps^t\geq \delta^{O_r(1)}$, the proof is complete.
\end{proof}

\begin{lemma} \label{lem:dependentrc for pairs}
    Let $0<\eps,\delta<1$ such that $\delta\geq 2\eps^4$. Let $G=(A,B)$ be a bipartite graph with $e(G)\geq \eps |A||B|$. Then there is a set $S\subset A$ of size at least $\frac{1}{2}\eps^4|A|$ such that all but at most $\delta |S|^2$ pairs of vertices in $S$ have at least $\eps^3 |B|$ common neighbours. 
\end{lemma}

\begin{proof}
Let $k=2$, $t=4$ and $\gamma=\eps^3$. Then it is easy to see that $\delta\eps^{kt}\geq 2\gamma^t$. Hence, the result follows directly from Lemma \ref{lem:dependent random choice}.
\end{proof}

Finally, we will use the following version of the Chernoff bound (see, e.g., Theorem A.1.13 in \cite{alonspencer}).

\begin{lemma} \label{lem:chernoff}
    Let $X=\sum_{i=1}^n X_i$, where $X_i=1$ with probability $p_i$ and $X_i=0$ with probability $1-p_i$, and all $X_i$ are independent. Let $\mu=\mathbb{E}[X]=\sum_{i=1}^n p_i$. Then for any $0<\delta<1$,
    $$\mathbb{P}(X\leq (1-\delta)\mu)\leq e^{-\mu \delta^2/2}.$$
\end{lemma}

\subsection{The absorption process} \label{sec:absorption process}

We now define our absorbing structure.

\begin{definition} \label{def:absorber}
	Let $G$ be an edge-coloured complete graph and let $F$ be a bipartite subgraph with parts $X$ and $Y$. We say that $F$ is $(\eta,\theta)$-good if there exist pairwise disjoint sets $Y_1,\dots,Y_t\subset Y$ such that the following hold.
	\begin{enumerate}
		\item $F$ is monochromatic.
		\item $F$ is isomorphic to an element of $\mathcal{F}$.
		\item For every $i\in [t]$, $|Y_i|\geq \eta |Y|$.
		\item $|Y\setminus \bigcup_{i\in [t]} Y_i|\leq \theta |Y|$. \label{cond:almost cover}
		\item For any $i\in [t]$ and any distinct $y_0,y_3\in Y_i$, there exist at least $\eta|Y_i|$ pairwise disjoint pairs $(y_1,y_2)\in Y_i^2$ of distinct vertices such that the edges between $y_0$ and $N_F(y_1)$, the edges between $y_1$ and $N_F(y_2)$ and the edges between $y_2$ and $N_F(y_3)$ are all of the same colour as $F$. \label{cond:switching}
	\end{enumerate}
\end{definition}

The key result of this subsection is that good subgraphs can indeed be used to absorb certain sets.

\begin{lemma} \label{lem:good partition absorbs}
	Let $\eta,\theta,K,C>0$ satisfy $\eta\geq \exp(-C\Delta)$, $\theta\geq \exp(-\exp(C\Delta))$, $K\leq \exp(\exp(C\Delta))$. Then there exists $C'=C'(C,r)$ such that the following holds.
	
	Let $G$ be an $r$-edge coloured complete graph and let $F=(X,Y)$ be an $(\eta,\theta)$-good subgraph in colour red. Let $Z$ be a subset of $V(G)$ such that $|Z|\leq K|Y|$ and for every $z\in Z$, the number of $y\in Y$ with $N_F(y)\subset N_{\textrm{red}}(z)$ is at least $2\theta |Y|$.
	
	Then $G[X\cup Y\cup Z]$ has a monochromatic $\FF$-tiling of size at most $\exp(C'\Delta)$.
\end{lemma}

To prove this result, we use the following technical lemma.

\begin{lemma} \label{lem:switch matching}
    Let $\eta,\theta,K,C>0$ satisfy $\eta,\theta\geq \exp(-\exp(C\Delta))$ and $K\leq \exp(\exp(C\Delta))$. Then there exists a constant $C'=C'(C,r)$ such that the following is true.
    
    Let $G$ be an $r$-edge coloured complete graph, let $Y$ and $Z$ be disjoint subsets of $V(G)$ and let $\sim$ be a binary relation defined over $Y\times (Y\cup Z)$ (i.e. a subset of $Y\times (Y\cup Z)$) which satisfies that $y\sim y$ for all $y\in Y$. Assume that for any distinct $x,y\in Y$, there exist at least $\eta|Y|$ pairwise disjoint pairs $(z,w)\in Y^2$ of distinct vertices such that $x\sim z$, $z\sim w$ and $w\sim y$. Assume also that $|Z|\leq K|Y|$ and that for every $z\in Z$, there are at least $\theta |Y|$ vertices $y\in Y$ such that $y\sim z$.
    
    Then there exists an injective map $f:Y\rightarrow Y\cup Z$ such that $y\sim f(y)$ for all $y\in Y$ and $G[Y\cup Z\setminus f(Y)]$ has a monochromatic $\FF$-tiling of size at most $\exp(C'\Delta)$.
\end{lemma}

\begin{proof}
Let $t=\min(\frac{\eta}{100} |Y|,\theta |Y|,\frac{1}{32\Delta r^{\Delta}}|Y|,|Z|)$. Applying Corollary \ref{cor:mono cover}, we find that there exists a collection of at most $64\Delta r^{\Delta}(\log(|Z|/t)+2)$ pairwise vertex-disjoint monochromatic subgraphs in $G[Z]$, each isomorphic to an element of $\mathcal{F}$, whose union is $Z\setminus T$ for some $T\subset Z$ of size at most $t$.

Since $|T|\leq t\leq \theta |Y|$ and for every $u\in T$ there are at least $\theta |Y|$ vertices $y\in Y$ with $y\sim u$, we can find an injective map $g:T\rightarrow Y$ such that $g(u)\sim u$ for all $u\in T$. Also, $|T|\leq t\leq \frac{1}{32\Delta r^{\Delta}}|Y|$, so by Corollary~\ref{cor:big mono graph}, there exists a monochromatic copy of $F_{|T|}$ in $G[Y]$. Let $R=\{x_1,x_2,\dots,x_{|T|}\}$ be the vertex set of this copy of $F_{|T|}$ and let $g(T)=\{y_1,y_2,\dots,y_{|T|}\}\subset Y$. After reordering the vertices if necessary, we may assume that there exists $0\leq \ell\leq |T|$ such that $x_i=y_i$ for all $\ell<i\leq |T|$ but $\{x_1,\dots,x_{\ell}\}\cap \{y_1,\dots,y_{\ell}\}=\emptyset$. Since $\ell\leq |T|\leq t\leq \frac{\eta}{100}|Y|$, we can choose pairwise disjoint vertices $z_1,w_1,z_2,w_2,\dots,z_{\ell},w_{\ell}$ in $Y\setminus \{x_1,x_2,\dots,x_{\ell},y_1,y_2,\dots,y_{\ell}\}$ such that for every $1\leq i\leq \ell$, $x_i\sim z_i$, $z_i\sim w_i$ and $w_i\sim y_i$.

Now define $f(v)=v$ for every $v\in Y\setminus (\bigcup_{i\leq \ell} \{x_i,y_i,z_i,w_i\}\cup \bigcup_{\ell<i\leq |T|} \{x_i\})$, let $f(x_i)=g^{-1}(y_i)$ for every $\ell<i\leq |T|$ and let $f(x_i)=z_i$, $f(z_i)=w_i$, $f(w_i)=y_i$ and $f(y_i)=g^{-1}(y_i)$ for every $i\leq \ell$. Then $v\sim f(v)$ holds for all $v\in Y$. Moreover, $f$ is a bijection between $Y$ and $(Y\setminus R)\cup T$, so $Y\cup Z\setminus f(Y)=R\cup (Z\setminus T)$. Hence $G[Y\cup Z\setminus f(Y)]$ has a monochromatic $\FF$-tiling of size at most $64\Delta r^{\Delta}(\log(|Z|/t)+2)+1$. It is not hard to see that the choice of $t$ implies the lemma. 
\end{proof}

\begin{proof}[Proof of Lemma \ref{lem:good partition absorbs}]
	First note that we may assume (by replacing $Z$ with $Z\setminus (X\cup Y)$ if necessary) that $Z$ is disjoint from $X\cup Y$.
	
	Let $Y_1,\dots,Y_t\subset Y$ witness the $(\eta,\theta)$-goodness of $F$. Fix some $z\in Z$. By assumption, there are at least $2\theta|Y|$ vertices $y\in Y$ with $N_F(y)\subset N_{\textrm{red}}(z)$. By condition (\ref{cond:almost cover}) of Definition \ref{def:absorber}, we have that the number of $y\in \cup_i Y_i$ such that $N_F(y)\subset N_{\textrm{red}}(z)$ is at least $\theta |Y|$. Thus, there exists $i\in [t]$ such that the number of $y\in Y_i$ with $N_F(y)\subset N_{\textrm{red}}(z)$ is at least $\theta |Y_i|$.
	
	Hence, we can obtain a partition $Z=\cup_{i=1}^t Z_i$ such that for every $z\in Z_i$ the number of $y\in Y_i$ with $N_F(y)\subset N_{\textrm{red}}(z)$ is at least $\theta |Y_i|$.
	
	For every $i\in [t]$, we define a binary relation $\sim_i$ on $Y_i\times (Y_i\cup Z_i)$ as follows. For $u\in Y_i$ and $v\in Y_i\cup Z_i$, we take $u\sim_i v$ if and only if $N_F(u)\subset N_{\textrm{red}}(v)$. Since $F$ is monochromatic red, we have $u\sim_i u$ for all $u\in Y_i$. By condition (\ref{cond:switching}) from Definition \ref{def:absorber}, for any distinct $x,y\in Y_i$, there exist at least $\eta |Y_i|$ pairwise disjoint pairs $(z,w)\in Y_i^2$ of distinct vertices such that $x\sim_i z$, $z\sim_i w$ and $w\sim_i y$. Moreover, for every $z\in Z_i$ there are at least $\theta |Y_i|$ vertices $y\in Y_i$ such that $y\sim_i z$. Finally, $|Z_i|\leq |Z|\leq K|Y|\leq K\eta^{-1} |Y_i|$. It follows by Lemma~\ref{lem:switch matching} that there exists an injection $f_i:Y_i\rightarrow Y_i\cup Z_i$ with $y\sim_i f_i(y)$ for all $y\in Y_i$ such that $G[Y_i\cup Z_i\setminus f_i(Y_i)]$ has a monochromatic $\FF$-tiling of size at most $\exp(C''\Delta)$, where $C''$ depends only on $C$ and $r$.
	
	Since clearly $t\leq \eta^{-1}$, $G[\bigcup_{i=1}^t (Y_i\cup Z_i\setminus f_i(Y_i))]$ has a monochromatic $\FF$-tiling of size at most $\eta^{-1}\exp(C''\Delta)$. The remaining set of vertices in $X\cup Y\cup Z$ is precisely $X\cup (Y\setminus \cup_{i=1}^t Y_i)\cup (\cup_{i=1}^t f_i(Y_i))$. We claim that this is the vertex set of a red subgraph isomorphic to $F$. Indeed, for every $y\in Y_i$ we have $y\sim_i f_i(y)$, so the edges in $G$ between $f_i(y)$ and $N_F(y)$ are all red. Hence, we can replace every $y\in Y_i$ (for every $i$) in $F$ with $f_i(y)$ and we get another red subgraph isomorphic to $F$ with vertex set $X\cup (Y\setminus \cup_{i=1}^t Y_i)\cup (\cup_{i=1}^t f_i(Y_i))$.
\end{proof}

\subsection{Constructing the absorbers} \label{sec:construct absorber}

In view of the previous subsection, we want to find certain good subgraphs. The key condition in Definition \ref{def:absorber} is property (\ref{cond:switching}). The next lemma will be used to obtain subsets which satisfy this property.

\begin{lemma} \label{lem:find one good set}
For any $0<\eps<1/100$ there exists a positive constant $c=c(\eps)=\Omega(\eps^{27})$ such that the following is true. Let $H=(A,B)$ be a bipartite graph with $|A|\leq |B|$. Suppose that for every $x\in A$, $|N(x)|\geq \eps|B|$. Then there exists a set $S\subset A$ of size at least $c|A|$ and a matching $f:S\rightarrow B$ in $H$ such that for any distinct $x,y\in S$, there are at least $c|S|$ pairwise disjoint pairs $(z,w)\in S^2$ of distinct vertices such that $f(x)\in N(z)$, $f(z)\in N(w)$ and $f(w)\in N(y)$.
\end{lemma}

\begin{proof}
First note that by replacing $A$ with a suitable subset of size $\frac{1}{128}\eps^{13}|A|$, it suffices to prove that the conclusion of the lemma holds for some $c=\Omega(\eps^{14})$ assuming that $|A|/|B|\leq \frac{1}{128}\eps^{13}$. Moreover, we may assume that $|A|\geq C\eps^{-14}$ for some sufficiently large absolute constant $C$ (else the statement becomes trivial). 

Let $\delta$ be chosen so that $\delta+\frac{2\delta}{\eps^3}=1/3$. Then $\delta/8\geq 2\eps^4$, so by Lemma \ref{lem:dependentrc for pairs} (applied with $\delta/8$ in place of $\delta)$, there exists a set $S_0\subset A$ of size at least $\frac{1}{2}\eps^4|A|$ such that the number of ordered pairs $(x,y)\in S_0^2$ for which $|N(x)\cap N(y)|<\eps^3|B|$ is at most $\frac{\delta}{4}|S_0|^2$. Let $S_1$ be the collection of those vertices $x\in S_0$ for which there are at most $\frac{\delta}{2}|S_0|$ vertices $y\in S_0$ satisfying $|N(x)\cap N(y)|<\eps^3|B|$. By a double counting argument, $|S_0\setminus S_1|\leq |S_0|/2$, so $|S_1|\geq |S_0|/2$. Hence, for any $x\in S_1$ there are at most $\delta |S_1|$ vertices $y\in S_1$ for which $|N(x)\cap N(y)|< \eps^3|B|$. Let $B'$ be the set of vertices $u\in B$ for which $|N(u)\cap S_1|\geq \delta |S_1|$. Observe that for any $x\in S_1$, $e(N(x)\setminus B',S_1)\leq |B|\delta |S_1|$, so the number of vertices $y\in S_1$ with $|N(x)\cap N(y)\setminus B'|\geq \frac{\eps^3}{2}|B|$ is at most $\frac{|B|\delta |S_1|}{\frac{\eps^3}{2}|B|}=\frac{2\delta}{\eps^3}|S_1|$. It follows that for any $x\in S_1$, the number of vertices $y\in S_1$ with $|N(x)\cap N(y)\cap B'|<\frac{\eps^3}{2}|B|$ is at most $(\delta+\frac{2\delta}{\eps^3})|S_1|=|S_1|/3$. This implies that for any $x,y\in S_1$, there are at least $|S_1|/3$ vertices $z\in S_1$ such that $|N(x)\cap N(z)\cap B'|\geq \frac{\eps^3}{2}|B|$ and $|N(y)\cap N(z)\cap B'|\geq \frac{\eps^3}{2}|B|$.

\noindent \emph{Claim.} For each $x\in S_1$, let $g(x)$ be a random neighbour of $x$ in $B'$.
Then with probability at least $2/3$, for any $x,y\in S_1$, there are at least $\frac{1}{16} \eps^{10} |S_1|$ pairwise disjoint pairs $(z,w)\in S_1^2$ of distinct vertices such that $g(x)\in N(z)$, $g(z)\in N(w)$ and $g(w)\in N(y)$.

\noindent \emph{Proof of Claim.} Let us estimate the probability that such pairs do not exist for some fixed $x$ and $y$. Condition on the event $g(x)=u$. Since $u\in B'$, we have $|N(u)\cap S_1|\geq \delta |S_1|$. Moreover, for any $z\in N(u)\cap S_1$, there are at least $|S_1|/3$ vertices $w\in S_1$ such that $|N(z)\cap N(w)\cap B'|\geq \frac{\eps^3}{2}|B|$ and $|N(w)\cap N(y)\cap B'|\geq \frac{\eps^3}{2}|B|$. Thus, we can greedily find at least $\delta |S_1|/2$ pairwise disjoint pairs $(z,w)\in S_1^2$ of distinct vertices such that $z\in N(u)$, $|N(z)\cap N(w)\cap B'|\geq \frac{\eps^3}{2}|B|$ and $|N(w)\cap N(y)\cap B'|\geq \frac{\eps^3}{2}|B|$. All but at most two such pairs are disjoint from $\{x,y\}$. For such a pair $(z,w)$, the probability that $g(z)\in N(w)$ and $g(w)\in N(y)$ is at least $(\frac{\eps^3}{2})^2$. Moreover, for disjoint pairs $(z,w)$ these events are independent. Hence, by a standard application of the Chernoff bound (Lemma \ref{lem:chernoff}), the probability that the number of good pairs $(z,w)$ is less than $\frac{1}{2}\cdot \frac{\delta |S_1|}{2} \cdot (\frac{\eps^3}{2})^2$ is at most $\exp(-\Omega(\delta \eps^6 |S_1|))\leq \exp(-\Omega(\eps^9 |S_1|))$. Thus, by the union bound, using that $|S_1|\geq |S_0|/2 \geq \eps^4|A|/4\geq C'\eps^{-10}$ for some sufficiently large $C'$, we have that with probability at least $2/3$, for any $x,y\in S_1$, there are at least $\frac{1}{2}\cdot \frac{\delta |S_1|}{2} \cdot (\frac{\eps^3}{2})^2=\frac{1}{16}\delta \eps^6 |S_1|$ pairwise disjoint pairs $(z,w)\in S_1^2$ of distinct vertices such that $g(x)\in N(z)$, $g(z)\in N(w)$ and $g(w)\in N(y)$. Since $\delta>\eps^4$, the proof of the claim is complete. $\Box$

This is close to what we need, but $g$ is not necessarily injective. Let $S$ be a maximal subset of $S_1$ on which $g$ is injective and set $f(x)=g(x)$ for every $x\in S$. Note that $|S_1\setminus S|$ is at most the number of pairs $(x,y)\in S_1^2$ with $x\neq y$ and $g(x)=g(y)$. For any $x\in S_1$, we have $|N(x)\cap B'|\geq \frac{\eps^3}{2}|B|$, so for any $x\neq y$, the probability that $g(x)=g(y)$ is at most $\frac{2}{\eps^3 |B|}$. Hence, the expected number of pairs $(x,y)\in S_1^2$ with $x\neq y$ and $g(x)=g(y)$ is at most $|S_1|^2\cdot \frac{2}{\eps^3 |B|}\leq \frac{2}{\eps^3}|S_1|\frac{|A|}{|B|}\leq \frac{1}{64}\eps^{10} |S_1|$ by our assumption that $|A|/|B|\leq \frac{1}{128}\eps^{13}$.
Hence, with probability at least $1/2$, $|S_1\setminus S|$ is less than $\frac{1}{32}\eps^{10} |S_1|$. If this holds and there are at least $\frac{1}{16}\eps^{10} |S_1|$ pairwise disjoint pairs $(z,w)\in S_1^2$ of distinct vertices such that $f(x)\in N(z)$, $f(z)\in N(w)$ and $f(w)\in N(y)$, then there are at least $\frac{1}{32}\eps^{10} |S_1|$ pairwise disjoint pairs $(z,w)\in S^2$ of distinct vertices such that $f(x)\in N(z)$, $f(z)\in N(w)$ and $f(w)\in N(y)$.

Using the claim and the last paragraph, we find that with probability at least $1/6$, for any $x,y\in S$, there are at least $\frac{1}{32}\eps^{10} |S_1|\geq \frac{1}{32}\eps^{10} |S|$ pairwise disjoint pairs $(z,w)\in S^2$ of distinct vertices such that $f(x)\in N(z)$, $f(z)\in N(w)$ and $f(w)\in N(y)$. Since $|S|\geq |S_1|/2\geq |S_0|/4\geq \eps^4|A|/8$, we may indeed find a suitable $c=\Omega(\eps^{14})$.
\end{proof}

The next lemma allows us to obtain not just one suitable subset, but an almost-cover by such sets.

\begin{lemma} \label{lem:find many good sets}
    For any $0<\eps< 1/100$ and $0<\theta<1$, there exists a positive constant $\eta=\eta(\eps,\theta)=\Omega(\theta \eps^{27})$ such that the following is true. Let $H=(A,B)$ be a bipartite graph with $|A|\leq \frac{\eps}{2}|B|$. Suppose that for every $x\in A$, $|N(x)|\geq \eps|B|$. Then there exist pairwise disjoint sets $S_1,S_2,\dots,S_t\subset A$ and an injection $f:A\rightarrow B$ such that
    \begin{enumerate}
    \item For every $i\in [t]$, $|S_i|\geq \eta|A|$.
    \item For every $x\in A$, $f(x)\in N(x)$.
    \item For any $i\in [t]$ and any distinct $x,y\in S_i$, there exist at least $\eta|S_i|$ pairwise disjoint pairs $(z,w)\in S_i^2$ of distinct vertices such that $f(x)\in N(z)$, $f(z)\in N(w)$ and $f(w)\in N(y)$. \label{cond:switching in S_i}
    \item $|A\setminus \bigcup_{i\in [t]} S_i|\leq \theta|A|$.
\end{enumerate}
\end{lemma}

\begin{proof}
Let $c=c(\eps/2)$ be the constant provided by Lemma \ref{lem:find one good set} and take $\eta=\theta c$. Note that $\eta\leq c$.

We define the sets $S_1,S_2,\dots$ recursively. Suppose that we have already found pairwise disjoint sets $S_1,S_2,\dots,S_k$ and matchings $f_i:S_i\rightarrow B$ in $H$ such that $f_1(S_1),\dots,f_k(S_k)$ are pairwise disjoint, $|S_i|\geq \eta|A|$ for every $i\leq k$, and for every $i\leq k$ and distinct $x,y\in S_i$, there exist at least $\eta|S_i|$ pairwise disjoint pairs $(z,w)\in S_i^2$ of distinct vertices such that $f_i(x)\in N(z)$, $f_i(z)\in N(w)$ and $f_i(w)\in N(y)$. Set $A_k=A\setminus \bigcup_{i\leq k} S_i$ and $B_k=B\setminus \bigcup_{i\leq k} f_i(S_i)$. If $|A_k|\leq \theta|A|$, terminate the process. Else, consider the bipartite graph $H_k=H[A_k,B_k]$. Observe that $|B\setminus B_k|\leq |A|\leq \frac{\eps}{2}|B|$. Hence, every $x\in A_k$ has $|N_{H_k}(x)|\geq \frac{\eps}{2}|B|\geq \frac{\eps}{2}|B_k|$. Moreover, $|A|\leq |B|$ easily implies that $|A_k|\leq |B_k|$. Therefore, we can apply Lemma \ref{lem:find one good set} (with $H_k$ in place of $H$ and $\eps/2$ in place of $\eps$) to find $S_{k+1}\subset A_k$ of size at least $c|A_k|$ and a matching $f_{k+1}:S_{k+1}\rightarrow B_k$ in $H_k$ such that for any distinct $x,y\in S_{k+1}$, there are at least $c|S_{k+1}|$ pairwise disjoint pairs $(z,w)\in S_{k+1}^2$ of distinct vertices such that $f_{k+1}(x)\in N_{H_k}(z)$, $f_{k+1}(z)\in N_{H_k}(w)$ and $f_{k+1}(w)\in N_{H_k}(y)$. Note that $|S_{k+1}|\geq c|A_k|\geq c\theta|A|=\eta|A|$.

The process must eventually terminate and then we have pairwise disjoint sets $S_1,S_2,\dots,S_t\subset A$ of size at least $\eta|A|$ each such that $|A\setminus \bigcup_i S_i|\leq \theta|A|$. Define $f:A\rightarrow B$ as follows. For $x\in S_i$, set $f(x)=f_i(x)$. For every $x\in A\setminus \bigcup_i S_i$, define $f(x)$ arbitrarily in a way that $f(x)\in N_H(x)$ and $f$ remains injective. This is possible since for every $x\in A$, $|N_H(x)|\geq \eps |B|\geq |A|$. It is easy to verify that these choices satisfy the conditions of the lemma.
\end{proof}

\subsection{Finding good subgraphs} \label{sec:good subgraph}

Call a hypergraph $\mathcal{G}$ \emph{down-closed} if for every $e\in E(\mathcal{G})$ and $e'\subset e$, we have $e'\in E(\mathcal{G})$. The following lemma is a straightforward modification of Lemma 2.2 from \cite{FS09}, but we provide a proof for completeness. By an embedding of a multihypergraph $\mathcal{H}$ into a simple hypergraph $\mathcal{G}$, we just mean an embedding of $\mathcal{H}'$ into $\mathcal{G}$, where $\mathcal{H}'$ is the simple hypergraph obtained by replacing the multiedges of $\mathcal{H}$ with simple edges.

\begin{lemma} \label{lem:embed hypergraph}
    Let $\mathcal{H}$ be an $m$-vertex multihypergraph with maximum degree at most $\Delta$ in which every hyperedge has size at most $\Delta$. Let $0<\lambda<1/(2\Delta)$ and let $\mathcal{G}$ be a down-closed hypergraph with $n\geq 2m$ vertices and more than $(1-\lambda^{\Delta})\binom{n}{\Delta}$ hyperedges of size $\Delta$. Then there are at least $(1-2\Delta \lambda)^m n(n-1)\cdots (n-m+1)$ labelled embeddings of $\mathcal{H}$ into $\mathcal{G}$.
\end{lemma}

\begin{proof}
    Let us call a set $S\subset V(\GG)$ of size at most $\Delta$ \emph{rich} if it is contained in more than $(1-\lambda^{\Delta-|S|})\binom{n-|S|}{\Delta-|S|}$ hyperedges of $\GG$ of size $\Delta$. Moreover, let us call a vertex $v\in V(\GG)\setminus S$ \emph{rich with respect to $S$} if $S\cup \{v\}$ is rich.
    
    We claim that for any rich set $S$ of size less than $\Delta$, there are at most $\lambda n$ vertices in $V(\GG)\setminus S$ which are not rich with respect to $S$. Indeed, if more than $\lambda n$ such vertices exist, then $S$ is contained in at least $\lambda n\cdot  \lambda^{\Delta-(|S|+1)}\binom{n-(|S|+1)}{\Delta-(|S|+1)}/(\Delta-|S|)\geq \lambda^{\Delta-|S|}\binom{n-|S|}{\Delta-|S|}$ sets of size $\Delta$ which are non-edges in $\GG$, which contradicts the assumption that $S$ is rich.
    
    Let $V(\HH)=\{v_1,\dots,v_m\}$ and write $V_k=\{v_1,\dots,v_k\}$ for every $0\leq k\leq m$. We will prove by induction on $k$ that for any $0\leq k\leq m$, there are at least $\prod_{0\leq i\leq k-1} (n-i-\Delta \lambda n)$ injective maps $f:V_k\rightarrow V(\GG)$ such that for every $e\in E(\HH)$, $f(e\cap V_k)$ is a rich set. For $k=0$, this follows from the fact that the empty set is rich. Now let $0\leq k\leq m-1$ and assume that $f:V_k\rightarrow V(\GG)$ is an injective map such that for $e\in E(\HH)$, $f(e\cap V_k)$ is a rich set. It suffices to prove that there are at least $n-k-\Delta \lambda n$ ways to extend $f$ to an injective map $f':V_{k+1}\rightarrow V(\GG)$ satisfying that for every $e\in E(\HH)$, $f'(e\cap V_{k+1})$ is a rich set. Note that $\HH$ has at most $\Delta$ edges containing $v_{k+1}$; let us call them $e_1,\dots,e_{\ell}$. Now if we choose $f'(v_{k+1})$ to be a vertex that does not belong to $f(V_k)$ and which is rich with respect to the set $f(e_i\cap V_k)$ for every $1\leq i\leq \ell$, then we get a desired extension of $f$. Thus, the number of choices for $f'(v_{k+1})$ which do not give a desired extension is at most $k+\ell\lambda n\leq k+\Delta \lambda n$. This completes the induction step.
    
    Since $m\leq n/2$, we have $n-i-\Delta \lambda n\geq (1-2\Delta \lambda)(n-i)$ for every $0\leq i\leq m-1$. Hence, there are at least $\prod_{0\leq i\leq m-1} (1-2\Delta \lambda)(n-i)=(1-2 \Delta \lambda)^{m} n(n-1)\dots (n-m+1)$ injective maps $f:V(\HH)\rightarrow V(\GG)$ such that $f(e)$ is rich for every $e\in E(\HH)$. For any such $f$ and any $e\in E(\HH)$, $f(e)$ is contained in at least one hyperedge of $\GG$, so $f(e)\in E(\GG)$, since $\GG$ is down-closed. Hence, any such $f$ gives a labelled embedding of $\HH$ into $\GG$.
\end{proof}

\begin{lemma} \label{lem:random embedding}
    Let $\mathcal{H}$ be an $m$-vertex multihypergraph with minimum degree at least one and maximum degree at most $\Delta$ in which every hyperedge has size at most $\Delta$. Let $U$ be a set of size $n\geq 16rm$ and let $R\subset U$ have size at least $\frac{n}{8r}$. Let $f$ be a uniformly random injective map from $V(\mathcal{H})$ to $U$. Then the probability that there are fewer than $\frac{e(\mathcal{H})}{\Delta^2 (32r)^{\Delta}}$ edges of $\mathcal{H}$ whose image under $f$ lie entirely in $R$ is at most $\exp(-\Omega(\frac{m}{\Delta^3 (16r)^{\Delta}}))$.
\end{lemma}

\begin{proof}
    Since every edge of $\mathcal{H}$ has size at most $\Delta$ and the maximum degree of $\mathcal{H}$ is at most $\Delta$, it follows that $\mathcal{H}$ has at least $e(\mathcal{H})/\Delta^2$ pairwise disjoint edges. Choose such edges $e_1,e_2,\dots,e_k$ with $k\geq e(\mathcal{H})/\Delta^2$. We can construct $f$ by taking an arbitrary list of the vertices of $\mathcal{H}$ and mapping them one by one to $U$, in each step randomly choosing one of the still available vertices. Take an ordering in which the vertices of $e_1$ come first, followed by the vertices of $e_2$ and so on. After the vertices of $e_k$, the remaining vertices are mapped in an arbitrary order. Note that $|V(\mathcal{H})|=m\leq \frac{n}{16r}\leq |R|/2$, so in each step at least $|R|/2\geq \frac{n}{16r}$ vertices of $R$ are still available. Hence, for every $i$, conditional on any mapping of the vertices of $e_1,\dots,e_{i-1}$, the probability that every vertex in $e_i$ gets mapped to $R$ is at least $(\frac{1}{16r})^{\Delta}$. Thus, the probability that fewer than $\frac{k}{(32r)^{\Delta}}$ of the edges $e_1,e_2,\dots,e_k$ get mapped to $R$ is upper bounded by the probability that the binomial random variable $\textrm{Bin}(k,(\frac{1}{16r})^{\Delta})$ takes value less than $k/(32r)^{\Delta}$. Again by a standard application of the Chernoff bound (Lemma \ref{lem:chernoff}), this probability is $\exp(-\Omega(k/(16r)^{\Delta}))$. Since $k\geq e(\mathcal{H})/\Delta^2$ and $e(\mathcal{H})\geq m/\Delta$ (the latter holds because of the minimum degree condition), the lemma follows.
\end{proof}

\begin{lemma} \label{lem:embed hypergraph carefully}
    There exists a constant $c=c(r)>0$ with the following property. Let $\mathcal{H}$ be an $m$-vertex multihypergraph with minimum degree at least one and maximum degree at most $\Delta$ in which every hyperedge has size at most $\Delta$. Let $\mathcal{G}$ be a down-closed hypergraph with $n\geq 16rm$ vertices and more than $\left(1-c^{\Delta^2}\right)\binom{n}{\Delta}$ hyperedges of size $\Delta$. Let $t\leq \exp(c^{\Delta}m)$ and for every $1\leq i\leq t$, let $R_i\subset V(\mathcal{G})$ have size at least $\frac{n}{8r}$.
    
    Then there exists an embedding of $\mathcal{H}$ into $\mathcal{G}$ such that for each $1\leq i\leq t$, the number of edges of $\mathcal{H}$ which are entirely in $R_i$ is at least $\frac{e(\mathcal{H})}{\Delta^2 (32r)^{\Delta}}$.
\end{lemma}

\begin{proof}
    Let $c$ be a sufficiently small positive constant, depending only on $r$. By Lemma \ref{lem:embed hypergraph} with $\lambda=c^{\Delta}$, there are at least $(1-2\Delta c^{\Delta})^mn(n-1)\cdots (n-m+1)$ labelled embeddings of $\mathcal{H}$ into $\mathcal{G}$. This means that a random injective map $f:V(\mathcal{H})\rightarrow V(\mathcal{G})$ defines a valid embedding with probability at least $(1-2\Delta c^{\Delta})^m$. Using $1-2\Delta c^{\Delta}\geq \exp(-4\Delta c^{\Delta})$, we get that $f$ is an embedding with probability at least $\exp(-4\Delta c^{\Delta}m)$. By Lemma \ref{lem:random embedding}, for any fixed $i$, the probability that $f$ maps fewer than $\frac{e(\mathcal{H})}{\Delta^2 (32r)^{\Delta}}$ edges of $\mathcal{H}$ to subsets of $R_i$ is at most $\exp(-\Omega(\frac{m}{\Delta^3 (16r)^{\Delta}}))$. Taking union bound over all $1\leq i\leq t$, we find that the probability that there exists $1\leq i\leq t$ for which $f$ maps fewer than $\frac{e(\mathcal{H})}{\Delta^2 (32r)^{\Delta}}$ edges of $\mathcal{H}$ to subsets of $R_i$ is at most $\exp(c^{\Delta}m-\Omega(\frac{m}{\Delta^3 (16r)^{\Delta}}))$. If $c$ is small enough, then this is less than $\exp(-4\Delta c^{\Delta}m)$, so there exists a choice for $f$ which defines a valid embedding and which maps at least $\frac{e(\mathcal{H})}{\Delta^2 (32r)^{\Delta}}$ edges of $\mathcal{H}$ to subsets of $R_i$ for each $1\leq i\leq t$.
\end{proof}

Observe that in the proof of Theorem \ref{thm:tiling with bipartite sequence}, we may assume that every $F_k\in \FF$ has at most one isolated vertex (else we may add edges, keep the graph bipartite and keep the maximum degree at most $\Delta$). In what follows, let $\FF$ satisfy this condition. For an edge-coloured graph $G$ and a vertex $w\in V(G)$, we write $N_{\textrm{red}}(w)$ for the set of vertices $u\in V(G)$ such that $wu$ is a red edge.

\begin{lemma} \label{lem:find good subgraph}
    Let $\eps=\Omega_r(1)^{\Delta}$ satisfy $\eps<1/100$, let $\theta=\frac{1}{2\Delta^2(32r)^{\Delta}}$ and let $c=c(r)$ be the constant from Lemma \ref{lem:embed hypergraph carefully}. Let $G$ be an $r$-edge coloured complete graph. Let $U$, $V$ and $W$ be subsets of $V(G)$ such that $U\cap V=\emptyset$ and $|U|\geq 100r^2$. Assume that for every $w\in W$, we have $|N_{\textrm{red}}(w)\cap U|\geq \frac{|U|}{8r}$. Moreover, assume that more than $(1-c^{\Delta^2})\binom{|U|}{\Delta}$ of the $\Delta$-sets in $U$ have at least $\eps|V|$ common red neighbours in $V$. Finally, assume that $|U|\leq \frac{\eps}{2}|V|$ and that $|W|\leq \exp(\frac{c^{\Delta}}{80r^2\Delta}|U|)$.
    
    \begin{enumerate}[label=(\alph*)]
        \item Then there exist $\eta=\Omega_r(1)^{\Delta}$ and a red $(\eta,\theta)$-good subgraph $F=(X,Y)$ in $G[U\cup V]$ with at most $|U|/(16r^2)$ vertices such that for every $w\in W$, the number of $y\in Y$ with $N_F(y)\subset N_{\textrm{red}}(w)$ is at least $2\theta |Y|$. \label{statement:good subgraph}
        
        \item Moreover, if $|W|\leq \exp(\exp(O_r(\Delta)))|U|$, then for any $Z\subset W$, $G[X\cup Y\cup Z]$ has a monochromatic $\FF$-tiling of size at most $O_r(1)^{\Delta}$. \label{statement:absorber}
    \end{enumerate}
\end{lemma}

\begin{proof}
    Let $k=\left\lfloor \frac{|U|}{16r^2}\right\rfloor$. Since $F_k$ has at most one isolated vertex, we can choose a bipartition $(X',Y')$ of $F_k$ such that every vertex in $X'$ has degree at least one. Also, since $F_k$ has at most one isolated vertex, $e(F_k)\geq \frac{k-1}{2}$, so since the maximum degree is at most $\Delta$, we have $|X'|,|Y'|\geq \frac{k-1}{2\Delta}$. Define a multihypergraph $\mathcal{H}$ whose vertex set is $X'$ and whose hyperedges are $N_{F_k}(y)$ for every $y\in Y'$ (with repetition). Since $F_k$ has maximum degree at most $\Delta$, it follows that $\mathcal{H}$ has maximum degree at most $\Delta$ and every hyperedge in $\mathcal{H}$ has size at most $\Delta$. Moreover, since every vertex in $X'$ has degree at least one in $F_k$, the minimum degree of $\mathcal{H}$ is at least one. Write $m=|V(\mathcal{H})|=|X'|$.
    
    Define a hypergraph $\mathcal{G}$ whose vertex set is $U$ and whose edges are those subsets of $U$ which have at least $\eps|V|$ common red neighbours in $V$. It is clear that $\mathcal{G}$ is down-closed. Writing $n=|V(\mathcal{G})|$, we have $n=|U|\geq 16r^2k\geq 16rm$. By the assumption on the red common neighbourhood of $\Delta$-sets in $U$, $\mathcal{G}$ has more than $(1-c^{\Delta^2})\binom{n}{\Delta}$ hyperedges of size $\Delta$.
    
    For every $w\in W$, let $R_w=N_{\textrm{red}}(w)\cap U$. By assumption, $|R_w|\geq \frac{n}{8r}$ holds for all $w\in W$. Since $m=|X'|\geq \frac{k-1}{2\Delta}\geq \frac{|U|}{80r^2\Delta}$, we have $|W|\leq \exp(\frac{c^{\Delta}}{80r^2\Delta}|U|)\leq \exp(c^{\Delta}m)$.
    
    By Lemma \ref{lem:embed hypergraph carefully}, there exists an embedding $g$ of $\mathcal{H}$ into $\mathcal{G}$ such that for every $w\in W$, the number of edges of $\mathcal{H}$ whose images are entirely in $R_w$ is at least $\frac{e(\mathcal{H})}{\Delta^2(32r)^{\Delta}}$.
    
    Let us define a bipartite graph $H$ as follows. The parts of $H$ are defined to be $A=E(\mathcal{H})$ and $B=V$. For every $e\in E(\mathcal{H})$, the neighbourhood of $e$ in $H$ is defined to be the common red neighbourhood inside $V$ of the vertices in $g(e)$ (note that $g(e)\subset U$). For any $e\in E(\mathcal{H})=A$, $g(e)\in E(\mathcal{G})$, so $g(e)$ has at least $\eps|V|$ common red neighbours in $V$, hence $|N_H(e)|\geq \eps |V|=\eps |B|$. Moreover, $|A|=|E(\mathcal{H})|=|Y'|\leq k\leq |U|\leq \frac{\eps}{2}|V|=\frac{\eps}{2}|B|$, so we may apply Lemma \ref{lem:find many good sets}. Let pairwise disjoint sets $S_1,\dots,S_t\subset A$ and an injection $f:A\rightarrow B$ satisfy the four properties in Lemma \ref{lem:find many good sets} with some $\eta=\Omega(\theta \eps^{27})$. Clearly, $\eta=\Omega_r(1)^{\Delta}$.
    
    Let $X=g(X')\subset U$ and let $Y=f(A)\subset V$. Define a bipartite graph $F$ with parts $X$ and $Y$ as follows. For every $y\in Y$, let the neighbourhood of $y$ in $F$ be $g(f^{-1}(y))\subset X$. We claim that $F$ is $(\eta,\theta)$-good. It is clear by the definition that $F$ is monochromatic red and is isomorphic to $F_k$. For each $i$, let $Y_i=f(S_i)$. Clearly, $|Y_i|=|S_i|\geq \eta |A|=\eta |Y|$. Furthermore, $|Y\setminus \bigcup_i Y_i|=|A\setminus \bigcup_i S_i|\leq \theta |A|=\theta |Y|$. It remains to check property~(\ref{cond:switching}) from Definition \ref{def:absorber}. Let $y_0,y_3$ be distinct vertices in $Y_i$ for some $i\in [t]$. Let $e_0=f^{-1}(y_0)$ and let $e_3=f^{-1}(y_3)$. By condition (\ref{cond:switching in S_i}) from Lemma \ref{lem:find many good sets}, there exist at least $\eta |S_i|$ pairwise disjoint pairs $(e_1,e_2)\in S_i^2$ of distinct vertices such that for each $0\leq b\leq 2$, $f(e_b)\in N_H(e_{b+1})$. For $j\in \{1,2\}$, let $y_j=f(e_j)$. It suffices to prove that for every $0\leq b\leq 2$, the edges between $y_b$ and the elements of $N_F(y_{b+1})$ are all red. By definition, $N_F(y_{b+1})=g(f^{-1}(y_{b+1}))=g(e_{b+1})$. But $f(e_b)\in N_H(e_{b+1})$ means precisely that every edge between $y_b=f(e_b)$ and the elements of $N_F(y_{b+1})=g(e_{b+1})$ is red, completing the proof that $F$ is $(\eta,\theta)$-good.
    
    Now let $w\in W$. The number of $y\in Y$ with $N_F(y)\subset R_w$ is equal to the number of edges $e\in E(\mathcal{H})$ such that $g(e)\subset R_w$. We know that this number is at least $\frac{e(\mathcal{H})}{\Delta^2(32r)^{\Delta}}=2\theta e(\mathcal{H})=2\theta |Y|$. Since $R_w=N_{\textrm{red}}(w)\cap U$, the proof of \ref{statement:good subgraph} is complete.
    
    For part \ref{statement:absorber}, we shall use Lemma \ref{lem:good partition absorbs}. If suffices to verify that $|W|\leq K|Y|$ holds for some $K\leq \exp(\exp(O_r(\Delta)))$. However, $|Y|=|Y'|\geq \frac{k-1}{2\Delta}\geq \frac{|U|}{80r^2\Delta}$, so the required inequality follows from $|W|\leq \exp(\exp(O_r(\Delta)))|U|$.
\end{proof}

The next lemma follows fairly easily by a repeated application of Lemma \ref{lem:find good subgraph} \ref{statement:absorber}. In an edge-coloured graph $G$ whose colours are labelled by the first $r$ positive integers, we write $N_i(w)$ for the set of vertices $u\in V(G)$ such that the edge $wu$ has colour $i$.

\begin{lemma} \label{lem:combine absorbers}
    Let $\eps=\Omega_r(1)^{\Delta}$ satisfy $\eps<1/100$ and let $c=c(r)$ be the constant from Lemma \ref{lem:embed hypergraph carefully}. Let $G$ be an $r$-edge coloured complete graph with the colours being the first $r$ positive integers. Let $k\in [r]$ and let $U$, $V_1,\dots,V_k$ and $W$ be subsets of $V(G)$ such that $U\cap V_i=\emptyset$ for each $1\leq i\leq k$. Assume that for every $w\in W$, there exists some $i\in [k]$ such that $|N_{i}(w)\cap U|\geq \frac{|U|}{4r}$. Moreover, assume that for each $i\in [k]$, fewer than $4^{-\Delta}c^{\Delta^2}\binom{|U|}{\Delta}$ sets of size $\Delta$ in $U$ have fewer than $\eps|V_i|$ common neighbours in colour $i$ within $V_i$. Finally, assume that $|W|\leq \exp(\exp(O_r(\Delta)))|U|$ and that for each $i\in [k]$, $|U|\leq \frac{\eps}{8}|V_i|$.
        
    Then, there is a set $D\subset V(G)$ of size at most $|U|/(16r)$ such that for any $Z\subset W$, $G[D\cup Z]$ has a monochromatic $\FF$-tiling of size at most $O_r(1)^{\Delta}$.
\end{lemma}

\begin{proof}
    First suppose that $|U|<200r^2\Delta$ or $|W|>\exp(\frac{c^{\Delta}}{160r^2\Delta}|U|)$. Then, since $|W|\leq \exp(\exp(O_r(\Delta)))|U|$, it follows that $|U|\leq \exp(\exp(O_r(\Delta)))$, and hence also $|W|\leq \exp(\exp(O_r(\Delta)))$. Now note that for any $Z\subset W$, $Z$ has size at most $\exp(O_r(1)^{\Delta})$ and therefore Corollary \ref{cor:mono cover} implies that $G[Z]$ has a monochromatic $\FF$-tiling of size at most $O_r(1)^{\Delta}$. Hence, in this case we can take $D=\emptyset$.
    
    Now assume that $|U|\geq 200r^2\Delta$ and $|W|\leq \exp(\frac{c^{\Delta^2}}{160r^2\Delta}|U|)$.
    
    For each $i\in [k]$, let $W_i$ be the set of vertices $w\in W$ for which $|N_i(w)\cap U|\geq \frac{|U|}{4r}$.

    \noindent \emph{Claim.} There exist pairwise disjoint sets $D_1,D_2,\dots,D_k\subset V(G)$ of size at most $|U|/(16r^2)$ each, such that for every $1\leq i\leq k$ and every $Z\subset W_i$, $G[D_i\cup Z]$ has a monochromatic $\FF$-tiling of size at most $O_r(1)^{\Delta}$.
    
    \noindent \emph{Proof of Claim.} We construct the sets one by one. Suppose that suitable $D_1,D_2,\dots,D_{i-1}$ have already been found. Let $\tilde{U}=U\setminus (D_1\cup \dots \cup D_{i-1})$ and $\tilde{V_i}=V_i\setminus (D_1\cup \dots \cup D_{i-1})$. We shall apply Lemma \ref{lem:find good subgraph} \ref{statement:absorber} with $\tilde{U}$ in place of $U$, $\tilde{V_i}$ in place of $V$ and $W_i$ in place of $W$. Since $|D_1\cup \dots \cup D_{i-1}|\leq r\cdot |U|/(16r^2)=|U|/(16r)$, it follows that every $w\in W_i$ satisfies that $|N_i(w)\cap \tilde{U}|\geq \frac{|U|}{8r}\geq \frac{|\tilde{U}|}{8r}$. Moreover, since $|\tilde{U}|\geq 2\Delta$ and $|\tilde{U}|/|U|\geq 1/2$, we have $\frac{\binom{|\tilde{U}|}{\Delta}}{\binom{|U|}{\Delta}}\geq (\frac{|\tilde{U}|}{2|U|})^{\Delta}\geq 4^{-\Delta}$, so there are fewer than $c^{\Delta^2}\binom{|\tilde{U}|}{\Delta}$ sets of $\Delta$ vertices in $\tilde{U}$ which have fewer than $\eps|V_i|$ common neighbours in colour $i$ inside $V_i$. Since $|D_1\cup \dots \cup D_{i-1}|\leq |U|/(16r)$ and $|U|\leq \frac{\eps}{8}|V_i|$, it follows that $|D_1\cup \dots \cup D_{i-1}|\leq \frac{\eps}{2}|V_i|$. Hence, there are fewer than $c^{\Delta^2}\binom{|\tilde{U}|}{\Delta}$ sets of $\Delta$ vertices in $\tilde{U}$ which have fewer than $\frac{\eps}{2}|\tilde{V_i}|$ common neighbours in colour $i$ inside $\tilde{V_i}$. Also, $|\tilde{U}|\leq |U|\leq \frac{\eps}{8}|V_i|\leq \frac{\eps}{4}|\tilde{V_i}|$. Moreover,
	$|\tilde{U}|\geq |U|/2\geq 100r^2$, $|W_i|\leq |W|\leq \exp(\frac{c^{\Delta}}{160r^2\Delta}|U|)\leq \exp(\frac{c^{\Delta}}{80r^2\Delta}|\tilde{U}|)$ and $|W_i|\leq \exp(\exp(O_r(\Delta)))|\tilde{U}|$.
    Thus, we can apply Lemma \ref{lem:find good subgraph} \ref{statement:absorber} (with $\eps/2$ in place of $\eps$), and we can take $D_i=X\cup Y$ for the sets $X$ and $Y$ provided by that lemma. This completes the proof of the claim. $\Box$
    
    Define $D=D_1\cup \dots \cup D_k$. Clearly, $|D|\leq k|U|/(16r^2)\leq |U|/(16r)$. Let $Z\subset W$. Define $Z_1=(Z\setminus D)\cap W_1$, $Z_2=(Z\setminus (D\cup W_1))\cap W_2$, \dots, $Z_k=(Z\setminus (D\cup W_1\cup \dots \cup W_{k-1}))\cap W_k$. Then $Z_1,\dots,Z_k$ partition $Z\setminus D$ and $Z_i\subset W_i$ holds for every $i$. By the claim above, for each $1\leq i\leq k$, $G[D_i\cup Z_i]$ has a monochromatic $\FF$-tiling of size $O_r(1)^{\Delta}$. But $D_1\cup Z_1,\dots,D_k\cup Z_k$ partition $D\cup Z$, so $G[D\cup Z]$ has a monochromatic $\FF$-tiling of size $O_r(1)^{\Delta}$.
\end{proof}

\subsection{Completing the proof of Theorem \ref{thm:tiling with bipartite sequence}} \label{sec:completing the proof}

\begin{lemma} \label{lem:induction step}
    Let $C>1$ be the constant from Lemma \ref{lem:k-set drc} and let $c$ be the constant from Lemma \ref{lem:embed hypergraph carefully}. Let $k\in [r]$. Let $\eps=\frac{1}{2^k r^{\Delta}}$, $\eps'=\frac{1}{2^{k+1}r^{\Delta}}$, $\delta=\exp(-(100c^{-1}Cr\Delta)^{2(r-k)+3})$ and $\delta'=\exp(-(100c^{-1}Cr\Delta)^{2(r-(k+1))+3})$. Let $G$ be an $r$-edge coloured complete graph with the colours labelled by the first $r$ positive integers. Let $A,B,V_1,V_2,\dots,V_{k-1}\subset V(G)$ and $K=\exp(\exp(O_r(\Delta)))$ satisfy the following properties.
    \begin{enumerate}
        \item $A$ is disjoint from $B\cup V_1\cup \dots \cup V_{k-1}$. \label{cond:l17cond1}
        \item There are at least $\frac{1}{r}|A||B|$ edges of colour $k$ between $A$ and $B$.
        \item For every $i\in [k-1]$, there are at most $\delta \binom{|A|}{\Delta}$ sets of size $\Delta$ in $A$ which have fewer than $\eps|V_i|$ common neighbours in colour $i$ inside $V_i$.
        \item $|A|\leq |B|\leq K|A|$ and $|A|\leq |V_i|$ for all $i\in [k-1]$. \label{cond:l17cond4}
    \end{enumerate}
    
    Then there are pairwise disjoint sets  $A'\subset A$, $B'\subset A\cup B$ and $D$ such that for any $Z\subset (A\cup B)\setminus (A'\cup B')$, $G[D\cup Z]$ has a monochromatic $\FF$-tiling of size at most $O_r(1)^{\Delta}$; and either $|A'\cup B'|\leq 2$ or there exist $V_1',V_2',\dots,V_k'\subset V(G)$ and $K'=\exp(\exp(O_r(\Delta)))$ such that the following hold.
    \begin{enumerate}[label=(\roman*)]
        \item $A'$ is disjoint from $B'\cup V_1'\cup \dots \cup V_k'$.
        \item There is some $j\in [r]\setminus [k]$ such that there are at least $\frac{1}{r}|A'||B'|$ edges of colour $j$ between $A'$ and $B'$.
        \item For every $i\in [k]$, there are at most $\delta' \binom{|A'|}{\Delta}$ sets of size $\Delta$ in $A'$ which have fewer than $\eps'|V_i'|$ common neighbours in colour $i$ inside $V_i'$.
        \item $|A'|\leq |B'|\leq K'|A'|$ and $|A'|\leq |V_i'|$ for all $i\in [k]$.
        \item $D$ is disjoint from $A'\cup B'\cup V_1'\cup \dots \cup V_k'$.
    \end{enumerate}
\end{lemma}

\begin{proof}
    Let $\delta''=\exp(-(100c^{-1}Cr\Delta)^{2(r-k)+2})$. If $|A|< 4\Delta(\delta'')^{-C}$, then by Corollary \ref{cor:mono cover} we may take $A'=B'=D=\emptyset$. So assume that $|A|\geq 4\Delta(\delta'')^{-C}$.
    
    By Lemma~\ref{lem:k-set drc} applied to the graph formed by edges of colour $k$ between $A$ and $B$ (with $\delta''$ playing the role of $\delta$), there is a set $U\subset A$ of size at least $(\delta'')^{C}|A|$ such that at most $\delta'' |U|^{\Delta}$ sets of $\Delta$ vertices in $U$ have fewer than $\eps|B|$ common neighbours in colour $k$ inside $B$. Since $|U|\geq 2\Delta$, it follows that $|U|^{\Delta}/\binom{|U|}{\Delta}\leq (2\Delta)^{\Delta}$, so at most $(2\Delta)^{\Delta}\delta''\binom{|U|}{\Delta}$ sets of $\Delta$ vertices in $U$ have fewer than $\eps |B|$ common neighbours in colour $k$ inside $B$. By replacing $U$ with a random subset of size $\lceil (\delta'')^C |A|\rceil$, we may assume that $(\delta'')^{C}|A|\leq |U|\leq 2(\delta'')^{C}|A|$. Again, by $|U|\geq 2\Delta$, we have $\frac{\binom{|U|}{\Delta}}{\binom{|A|}{\Delta}}\geq (\frac{|U|}{2|A|})^{\Delta}\geq 2^{-\Delta}(\delta'')^{C\Delta}$. Note that $\frac{\delta}{2^{-\Delta}(\delta'')^{C\Delta}}\leq \delta''$, so for every $1\leq i\leq k-1$, there are at most $\delta''\binom{|U|}{\Delta}$ sets of size $\Delta$ in $U$ which have fewer than $\eps|V_i|$ common neighbours in colour $i$ inside $V_i$. Setting $V_k=B$, we get that for every $1\leq i\leq k$, there are at most $(2\Delta)^{\Delta}\delta''\binom{|U|}{\Delta}$ sets of $\Delta$ vertices in $U$ which have fewer than $\eps|V_i|$ common neighbours in colour $i$ inside $V_i$. Let $W$ be the set of vertices $w\in A\cup B$ for which there exists $i\in [k]$ with $|N_i(w)\cap U|\geq \frac{|U|}{4r}$. Note that $|W|\leq |A|+|B|\leq 2K|A|\leq 2K(\delta'')^{-C}|U|\leq \exp(\exp(O_r(\Delta)))|U|$. Moreover, for every $i\in [k]$, $|U|\leq 2(\delta'')^C|A|\leq \frac{\eps}{8}|A|\leq \frac{\eps}{8}|V_i|$.
    
    Since $(2\Delta)^{\Delta}\delta''< 4^{-\Delta}c^{\Delta^2}$, Lemma \ref{lem:combine absorbers} implies that there is a set $D\subset V(G)$ of size at most $|U|/(16r)$ such that for any $Z\subset W$, $G[D\cup Z]$ has a monochromatic $\FF$-tiling of size at most $O_r(1)^{\Delta}$.
    
    Let $S=(A\cup B)\setminus (D\cup W)$ and let $U'=U\setminus D$. If $|U'|<|S|/2$, let $A'=U'$; otherwise let $A'$ be a uniformly random subset of $U'$ of size $\lfloor |S|/2\rfloor$. In either case, let $B'=S\setminus A'$. It is easy to see that $A'$, $B'$ and $D$ are pairwise disjoint. Also, $S\subset A'\cup B'$, so $(A\cup B)\setminus (A'\cup B')\subset (A\cup B)\setminus S\subset D\cup W$. Let $Z\subset (A\cup B)\setminus (A'\cup B')$. By the above, we have $Z\subset D\cup W$, so $Z\setminus D\subset W$. Hence, $G[D\cup Z]=G[D\cup (Z\setminus D)]$ has a monochromatic $\FF$-tiling of size at most $O_r(1)^{\Delta}$. This means that we are done if $|A'\cup B'|\leq 2$, so let us assume that $|A'\cup B'|\geq 3$. Then necessarily $|S|>1$ and $A'\neq \emptyset$.
    
    For every $1\leq i\leq k$, let $V_i'=V_i\setminus D$. It is clear that $D$ is disjoint from $A'\cup B'\cup V_1'\cup \dots \cup V_k'$. Moreover, since $A'\subset A$ and $V_i'\subset V_i$ for all $i\in [k]$, it follows that $A'$ is disjoint from $B'\cup V_1'\cup \dots \cup V_k'$.
    
    Note that $|A'|\leq |S|/2\leq |B'|$. Also, if $|U'|<|S|/2$, then $|A'|=|U'|\geq |U|/2\geq \frac{(\delta'')^C}{2}|A|\geq \frac{(\delta'')^C}{2(K+1)}|A\cup B|\geq \frac{(\delta'')^C}{2(K+1)}|B'|$. On the other hand, if $|U'|\geq |S|/2$, then $|A'|=\lfloor |S|/2\rfloor$, whereas $|B'|\leq |S|$. Thus, in both cases, $|B'|\leq K'|A'|$ for some $K'=\exp(\exp(O_r(\Delta)))$. Finally, for every $1\leq i\leq k$, $|V_i'|\geq |V_i|/2\geq |A|/2$ and $|A'|\leq |U|\leq |A|/2$, so $|A'|\leq |V_i'|$.
    
    Recall that for every $1\leq i\leq k$, there are at most $(2\Delta)^{\Delta}\delta''\binom{|U|}{\Delta}$ sets of $\Delta$ vertices in $U$ which have fewer than $\eps|V_i|$ common neighbours in colour $i$ inside $V_i$. Moreover, $\binom{|U|}{\Delta}/\binom{|U'|}{\Delta}\leq 4^{\Delta}$ (since $|U'|\geq |U|/2\geq 2\Delta$) and $|D|\leq |U|\leq \frac{\eps}{8}|V_i|$, so there are at most $4^{\Delta}(2\Delta)^{\Delta}\delta''\binom{|U'|}{\Delta}$ sets of $\Delta$ vertices in $U'$ which have fewer than $\frac{\eps}{2}|V_i'|$ common neighbours in colour $i$ inside $V_i'$. Since $A'$ is a uniformly random subset of $U'$ of certain prescribed size, it follows that for each $1\leq i\leq k$, the expected number of sets of $\Delta$ vertices in $A'$ which have fewer than $\frac{\eps}{2}|V_i'|$ common neighbours in colour $i$ inside $V_i'$ is at most $(8\Delta)^{\Delta}\delta''\binom{|A'|}{\Delta}$. Hence, by Markov's inequality, for any $1\leq i\leq k$, the probability that the number of sets of $\Delta$ vertices in $A'$ which have fewer than $\frac{\eps}{2}|V_i'|$ common neighbours in colour $i$ inside $V_i'$ is at least $6k\cdot (8\Delta)^{\Delta}\delta''\binom{|A'|}{\Delta}$ is at most $\frac{1}{6k}$. Thus, by the union bound, it holds with probability at least $5/6$ that for every $1\leq i\leq k$, the number of sets of $\Delta$ vertices in $A'$ which have fewer than $\frac{\eps}{2}|V_i'|$ common neighbours in colour $i$ inside $V_i'$ is at most $6k\cdot (8\Delta)^{\Delta}\delta''\binom{|A'|}{\Delta}$. Note that $\delta'\leq 6k\cdot (8\Delta)^{\Delta}\delta''$.
    
    Now let $1\leq i\leq k$ and let $z\in S$. Then $z\not \in W$, so $|N_i(z)\cap U|<\frac{|U|}{4r}$. Moreover, $|U'|\geq \frac{3}{4}|U|$, so $|N_i(z)\cap U'|<\frac{|U'|}{3r}$. Hence, $\sum_{i\in [k], z\in S} |N_i(z)\cap U'|<\frac{k}{3r}|S||U'|$. Recall that either $A'=U'$ or $A'$ is a uniformly random subset of $U'$ of size $\lfloor |S|/2\rfloor$. In either case,
    $\mathbb{E}\left[\sum_{i\in [k], z\in S} |N_i(z)\cap A'|\right]<\frac{k}{3r}|S||A'|$.
    Thus, the probability that $\sum_{i\in [k], z\in S} |N_i(z)\cap A'|\geq \frac{k}{2r}|S||A'|$ is at most $2/3$. But $|S|\leq 2|B'|$, so then the probability that $\sum_{i\in [k], z\in B'} |N_i(z)\cap A'|\geq \frac{k}{r}|B'||A'|$ is also at most $2/3$.
    It follows from this and the previous paragraph that there exists an outcome for the random set $A'$ such that the number of edges between $A'$ and $B'$ with colour in $[k]$ is less than $\frac{k}{r}|A'||B'|$ and for each $i\in [k]$, the number of sets of $\Delta$ vertices in $A'$ which have fewer than $\eps'|V_i'|$ common neighbours in colour $i$ inside $V_i'$ is at most $\delta'\binom{|A'|}{\Delta}$. Then there is some $j\in [r]\setminus [k]$ such that there are at least $\frac{1}{r}|A'||B'|$ edges of colour $j$ between $A'$ and $B'$.
    
    This completes the proof of the lemma.
\end{proof}

The next result follows easily from an iterative application of Lemma \ref{lem:induction step}.

\begin{corollary} \label{cor:iterated absorbers}
    Let $G$ be an $r$-edge coloured complete graph. Then there exist $0\leq \ell\leq r$, pairwise disjoint sets $D_1,\dots,D_{\ell}\subset V(G)$ and sets $T_{\ell+1}\subset T_{\ell}\subset \dots \subset T_1=V(G)$ such that $|T_{\ell+1}|\leq 2$ and for any $1\leq i\leq \ell$ and $Z\subset T_i\setminus T_{i+1}$, $G[D_i\cup Z]$ has a monochromatic $\FF$-tiling of size $O_r(1)^{\Delta}$.
\end{corollary}

\begin{proof}
    Let $T_1=V(G)$. If $|T_1|\leq 2$, then we can take $\ell=0$. Else, let $A$ be an arbitrary subset of $V(G)$ of size $\lfloor |V(G)|/2 \rfloor$ and let $B=V(G)\setminus A$. Label the most frequent colour between $A$ and $B$ by the number~1. Then the conditions (\ref{cond:l17cond1})-(\ref{cond:l17cond4}) from Lemma \ref{lem:induction step} are satisfied for $k=1$. We can now repeatedly apply the lemma, after each step removing the subset $D$ from the graph, until we eventually (after at most $r$ steps) obtain $A',B'$ with $|A'\cup B'|\leq 2$. Let $\ell$ be the number of iterations of Lemma \ref{lem:induction step} before this happens. For $1\leq i\leq \ell$, let $D_i$ be the set $D$ obtained by Lemma \ref{lem:induction step} in the $i$th iteration, and let $T_{i+1}$ be the set $A'\cup B'$ obtained by Lemma \ref{lem:induction step} in the $i$th iteration.
\end{proof}

It is easy to deduce our main result from this.

\begin{proof}[Proof of Theorem \ref{thm:tiling with bipartite sequence}.]
    Let $G$ be an $r$-edge coloured complete graph. Choose $\ell$, $D_1,\dots,D_{\ell}$, $T_1,\dots,T_{\ell+1}$ according to Corollary \ref{cor:iterated absorbers}. For each $i\in [\ell]$, let $Z_i=T_i\setminus (T_{i+1}\cup D_1\cup \dots \cup D_{\ell})$. For each $i\in [\ell]$, $G[D_i\cup Z_i]$ has a monochromatic $\FF$-tiling of size $O_r(1)^{\Delta}$. Since the sets $D_1\cup Z_1,\dots,D_{\ell}\cup Z_{\ell},T_{\ell+1}\setminus (D_1\cup \dots \cup D_{\ell})$ partition $V(G)$ and $|T_{\ell+1}|\leq 2$, it follows that $G$ has a monochromatic $\FF$-tiling of size $O_r(1)^{\Delta}$.
\end{proof}

\section{Concluding remarks}

A graph is called $d$-degenerate if each of its subgraphs has a vertex of degree at most $d$. Proving a conjecture of Burr and Erd\H os \cite{BE75}, Lee \cite{Lee17} showed that the Ramsey number of an $n$-vertex $d$-degenerate graph is at most $c(d)n$. Hence, it is natural to wonder whether it is sufficient to assume that each $F_k\in \FF$ is $d$-degenerate (rather than that it has maximum degree $d$) to guarantee the boundedness of $\tau_r(\FF)$. However, this is not the case, even when each $F_k\in \FF$ is bipartite. In fact, it is not even enough that each $F_k$ is bipartite with maximum degree at most $d$ on one side. Indeed, it was observed by Pokrovskiy (see \cite{CM21} for a proof) that when $\FF=\{S_1,S_2,S_3,\dots\}$ is the collection of stars, $\tau_r(\FF)=\infty$ holds for any $r\geq 2$.

Hence, it is perhaps slightly surprising that using similar techniques as in this paper, we can prove the following result.

\begin{theorem}
    For any $r\in \mathbb{N}$, there is a $C_r$ such that the following is true. Let $\FF=\{K_1',K_2',K_3',\dots\}$, where $K_t'$ denotes the $1$-subdivision of $K_t$ and let $G$ be an $r$-edge coloured complete graph. Then $G$ has a monochromatic $\FF$-tiling of size at most $C_r$. 
\end{theorem}

As the proof would require several additional pages, we omit it.

We conclude by mentioning a question asked by Corsten and Mendon\c{c}a.

\begin{problem}[Corsten--Mendon\c{c}a \cite{CM21}]
    Is there a function $g:\mathbb{N}\rightarrow \mathbb{N}$ with $\lim_{n\rightarrow \infty} g(n)=\infty$ such that the following is true for all positive integers $r$ and $d$? If $\FF=\{F_1,F_2,\dots\}$ is a sequence of $d$-degenerate graphs with $|V(F_i)|=i$ and $\Delta(F_i)\leq g(i)$ for all~$i$, then $\tau_r(\FF)<\infty$.
\end{problem}

It would also be interesting to answer this question in the case where all $F_i$ are bipartite.

\bibliographystyle{abbrv}
\bibliography{bibliography}

\end{document}